\documentclass[11pt, reqno]{amsart}

\usepackage{enumitem}

\usepackage{geometry}
\usepackage{color}
\usepackage{amssymb}
\usepackage{amsmath}
\usepackage{arydshln}
\usepackage{hyperref}
\usepackage{bbm}
\usepackage{mathrsfs}

\usepackage[T1]{fontenc}
\usepackage{dsfont} 
\usepackage{tikz}
\usetikzlibrary{matrix}

\def\BibTeX{{\rm B\kern-.05em{\sc i\kern-.025em b}\kern-.08em
    T\kern-.1667em\lower.7ex\hbox{E}\kern-.125emX}}

\numberwithin{equation}{section}

\newtheorem{Theorem}{Theorem}[section]

\newtheorem{Corollary}[Theorem]{Corollary}
\newtheorem{Lemma}[Theorem]{Lemma}

\newtheorem{Definition}[Theorem]{Definition}

\begin{document}

\title{A topological proof of Sklar's theorem in arbitrary dimensions}

\author{Fred Espen Benth}
\address[Fred Espen Benth]{Department of Mathematics\\ University of Oslo, Norway}
\email[Fred Espen Benth]{fredb@math.uio.no}

\author{Giulia Di Nunno}
\address[Giulia Di Nunno]{Department of Mathematics\\ University of Oslo, Norway}
\email[Giulia Di Nunno]{giulian@math.uio.no}

\author{Dennis Schroers}
\address[Dennis Schroers]{Department of Mathematics\\ University of Oslo, Norway}
\email[Dennis Schroers]{dennissc@math.uio.no}

\date{\today}



\maketitle

\begin{abstract}
We prove Sklar's theorem in infinite dimensions via a topological argument and the notion of inverse systems. 
\end{abstract}
\section{Introduction}
Copulas are widely used and well known concepts in the realm of statistics and probability theory. The keystone of the theory is Sklar's theorem and there is a vast literature solely focussing on different proofs of this fundamental result.
Among others there are proofs based on the distributional transform in \cite{Ruschendorf2009} and \cite{Deheuvels2009} and earlier already in \cite{Moore1975}, based on mollifiers in \cite{MR2847456} or the constructive approach by the extension of subcopulas, as it was proved for the bivariate case in \cite{Schweizer1974} and for the general multivariate case in \cite{Sklar1996} or \cite{Carley2002}. 

The naive transfer of the subcopula-approach to an infinite-dimensional setting appears to be challenging, since, after the extension of the subcopulas corresponding to the finite-dimensional laws of an infinite-dimensional distribution, one would also have to check that this construction meets the necessary consistency conditions.
In contrast, and besides the approach via distributional transforms (as extended to an infinite dimensional setting in \cite{Benth2020}), a nonconstructive proof based on topological arguments in
 \cite{Durante2013} is naturally in tune with an infinite dimensional setting.
 
In this paper, we will therefore adopt this ansatz and prove Sklar's theorem in infinite dimensions by equipping the space of copulas with an inverse-limit topology that makes it compact and the operation between marginals and copulas induced by Sklar's theorem continuous. 
The compactness of copulas is described as "folklore" in \cite{MR2847456} for the finite dimensional case, which is why the transfer to arbitrary dimensions is desirable.

\section{Short Primer on Topological Inverse Systems}
We will frequently use the notation $\bar{\mathbb R}$ for the extended real line $[-\infty,\infty]$. 
 For any measure $\mu$ on a measurable space $(B,\mathcal B)$ and a measurable function $f:(B,\mathcal B)\to (A,\mathcal A)$ into another measurable space $(A,\mathcal A)$ we denote by $f_*\mu$ the pushforward measure 
 with respect to $f$ given by $f_*\mu(S):=\mu( f^{-1}(S))$ for all $S\in \mathcal A$. 
 For $I$ an arbitrary index set, $B=\bar{\mathbb{R}}^I$ and $\mathcal B=\otimes_{i\in I} \mathcal{B}(\bar{\mathbb R})$, we use the shorter notations $\pi_{J*}\mu=:\mu_J$ for a subset $J\subseteq I$ and $\pi_{\lbrace i\rbrace*}\mu=:\mu_i$ for an element $i\in I$, where $\pi_J$ denotes the canonical projection on $\mathbb{R}^J$.  If $J\subset I$ is finite, we denote the corresponding finite dimensional cumulative distribution functions by $F_{\mu_J}$ or $F_{\mu_i}$ respectively, where in the latter we used $J=\{i\}$. We use the notation $\mathcal I$ for the set consisting of all finite subsets of $I$. 
 Moreover, for a one-dimensional Borel measure $\mu_i$ on $\mathbb R$, we use the notation $F_{\mu_i}^{[-1]}$ for the quantile functions
\begin{equation}\label{Inverse Transform}
    F_{\mu_i}^{[-1]}(u) := \inf \left\lbrace x\in(-\infty,\infty) : F_{\mu_i}(x)\geq u\right\rbrace.
\end{equation}
 We will refer to the one dimensional distributions $\mu_i,i\in I$ and equivalently $F_{\mu_i},i\in I$ as marginals of the measure $\mu$. We denote the set of all probability measures  on $(\bar{\mathbb{R}}^I,\otimes_{i\in I}\mathcal{B}(\bar{\mathbb{R}}))$ by $\mathcal P(\bar{\mathbb{R}}^I)$.
Moreover, for two topological spaces $X,Y$ we write $X\cong Y$ if they are homeomorphic.

The remainder of the section is mainly based on \cite{MR2599132}. 
Let $X_{J}$ be a set for each $J\in \mathcal{I}$ and 
\begin{equation*}
    (P_{J_1,J_2}:X_{J_2}\to X_{J_1})\qquad \text{for } J_1\subseteq J_2,\text{ with } J_1,J_2\in\mathcal{I}
\end{equation*}
 a family of mappings, also called projections, such that 
\begin{itemize}
\item[(i)]$P_{J,J}=id_J$ is the identity mapping for all $J\in \mathcal{I}$, and
\item[(ii)] $P_{J_1,J_3}=P_{J_1,J_2}\circ P_{J_2,J_3}$ for all $J_1\subseteq J_2\subseteq J_3$ in $\mathcal{I}$.
\end{itemize} 
The system $$(X_J,P_{J_1,J_2},\mathcal{I}):=\left((X_J)_{J\in \mathcal{I}},((P_{J_1,J_2}:X_{J_2}\to X_{J_1})_{\overset{J_1\subseteq J_2}{J_1,J_2\in\mathcal{I}}})\right)
$$
is called an \textit{inverse system} (over the partially ordered set $\mathcal{I}$). 
 If $(X_{J},\tau_{J})$ are topological spaces for each $J\in \mathcal{I}$ and $(P_{J_1,J_2})$ are continuous for all $J_1\subseteq J_2$ with $J_1,J_2\in\mathcal{I}$, we call 
 $$(X_J,\tau_J,P_{J_1,J_2},J\in \mathcal{I}):=\left((X_J,\tau_J)_{J\in \mathcal{I}},((P_{J_1,J_2}:X_{J_2}\to X_{J_1})_{\overset{J_1\subseteq J_2}{J_1,J_2\in\mathcal{I}}})\right)$$ a \textit{topological inverse system}.
 A \textit{topological inverse limit} of this inverse system 
 is a space $X$ together with continuous mappings $P_J:X\mapsto X_J, J\in \mathcal I$, such that  $P_{J_1,J_2}P_{J_2}=P_{J_1}$ for all $J_1\subseteq J_2$ in $\mathcal{I}$ 
(that is, the mappings are \textit{compatible})
 and the following \textit{universal property} holds:
 Whenever there is a topological space $Y$, such that there are continuous mappings $(\psi_J:Y\to X_J)_{J\in \mathcal I}$ which are compatible, i.e., $P_{J_1,J_2}\psi_{J_2}=\psi_{J_1}$ for all $J_1\subseteq J_2$ in $\mathcal{I}$, then there exists a unique continuous mapping
 \begin{equation}\label{universal property of inverse limit}
      \Psi:Y\to X, 
      \end{equation}
      with the property $P_J\Psi=\psi_J$ for all $J\in \mathcal I$. We have that 
      \begin{equation}\label{D: Projective Limit}
\left\lbrace x=(x_J)_{J\in \mathcal{I}} \in \prod_{J\in \mathcal{I}}X_J:P_{J_1,J_2}(\pi_{J_2}(x))=\pi_{J_1}(x) \text{ for }J_1\subseteq J_2\right\rbrace\subseteq \prod_{J\in \mathcal I}X_J
\end{equation}
equipped with the subspace topology with respect to the product topology is an inverse limit of the topological inverse system, induced by the canonical projections $\pi_{J'} ((x_J)_{J\in\mathcal I})=x_{J'}$. 
Each topological inverse limit is homeomorphic to this space and therefore to every topological inverse limit
(See the proof of Theorem 1.1.1 in \cite{MR2599132}).
We write $\lim_{\leftarrow}X_J\subseteq \prod_{J\in\mathcal I}X_J$ for the inverse limit as a subset of the product space and we equip it 
throughout with the induced subspace topology. 
\begin{Lemma}\label{L: Closedness of the Inverse limit}
Let $(X_{J},\tau_{J},\pi_{J_1,J_2})$ be a topological inverse system (over the poset $\mathcal{I}$) of Hausdorff spaces. Then $\lim_{\leftarrow}X_J$ is a closed subset of $\prod_{J\in\mathcal I}X_J$ with respect to the product topology.
\end{Lemma}

\begin{proof}
See \cite[Lemma 1.1.2]{MR2599132}.
\end{proof}

\begin{Lemma}\label{L: Surjectivity of the induced mapping}
Let $X$ be a compact Hausdorff space and $(X_{J},\tau_{J},\pi_{J_1,J_2})$ be a topological inverse system of compact Hausdorff spaces. Let $\psi_J:X\to X_J,\, J\in\mathcal I$ be a family of compatible surjections and $\Psi$ the induced mapping. Then either $\lim_{\leftarrow}X_J=\emptyset$ or $\Psi(X)$ is dense in $\lim_{\leftarrow}X_J$. 
\end{Lemma}
\begin{proof}
See \cite[Corollary 1.1.7]{MR2599132}.
\end{proof}


\section{Copulas and Sklar's Theorem }

As they are cumulative distribution functions, copulas in finite dimension have a one-to-one correspondence to probability measures. 
In infinite dimensions we will therefore work with the notion of copula measures as introduced in \cite{Benth2020}. 
\begin{Definition}\label{T: Consistent Copulas}
A copula measure (or simply copula) on $\bar{\mathbb{R}}^I$ is a probability measure $C\in\mathcal P(\bar{\mathbb R}^I)$, such that its marginals $C_i$ are uniformly distributed on $[0,1]$.
 We will denote the space of copula measures on $\bar{\mathbb R}^I$ by $\mathcal C(\bar{\mathbb R}^I)$.
\end{Definition}

 Sklar's theorem as stated below was proved in \cite{Benth2020} by following the arguments for the finite dimensional assertion in \cite{Ruschendorf2009}. Here we give an alternative proof for the infinite dimensional setting using a topological argument as in \cite{Durante2013}.
\begin{Theorem}[Sklar's Theorem]\label{T: Sklar in infinite dimensions}
Let $\mu \in\mathcal P (\bar{
\mathbb R}^I)$ be a probability measure with marginal one-dimensional distributions $\mu_i, i\in I$.
There exists a copula measure $C$, such that for each $J\in \mathcal I$, we have
\begin{equation}\label{Sklar Property}
F_{C_J}\left(\left(F_{\mu_{j}}(x_{j})\right)_{j\in J}\right)=F_{\mu_J}\left((x_{j})_{j\in J}\right)
\end{equation} 
for all $(x_{j})_{j\in J}\in \bar{\mathbb{R}}^{ J}$. Moreover, $C$ is unique if $F_{\mu_{i}}$ is continuous for each $i\in I$.
Vice versa, let $C$ be a copula measure on $\bar{\mathbb{R}}^I$ and let $(\mu_i)_{i\in I}$ be a collection of (one-dimensional) Borel probability measures over $\bar{\mathbb R}$. 
Then there exists a unique probability measure $\mu\in\mathcal P(\bar{\mathbb R}^I)$, such that 
\eqref{Sklar Property} holds. 
\end{Theorem}

\section{Topological Properties of Copulas and a Proof of Sklar's Theorem }

  The collection $(\mathcal P(\bar{\mathbb R}^J), J\in \mathcal I)$, where each $\mathcal P(\bar{\mathbb R}^J)$ is considered as a topological space  with the topology of weak convergence, is a topological inverse system with the projections
$\pi_{J_1,J_2}(\mu_{J_2})=(\mu_{J_2})_{J_1}$ for $\mu_{J_2}\in \mathcal P(\bar{\mathbb R}^{J_2})$ and $J_1,J_2\in \mathcal I$, $J_1\subseteq J_2$. Moreover, observe that each $\mathcal P(\bar{\mathbb R}^J)$ is a Hausdorff space, since it is metrizable by the Prohorov metric (c.f. \cite[Theorem 4.2.5]{Schweitzer2006}).
The space $\lim_{\leftarrow}\mathcal P(\bar{\mathbb R}^J)\subset \prod_{J\in\mathcal I}\mathcal P(\bar{\mathbb R}^J)$ of consistent families of probability measures is a topological inverse limit, equipped with the corresponding inverse limit topology.
The space of probability measures on $\otimes_{i\in I}\mathcal B(\mathbb R)$ has via its finite-dimensional distributions a one-to-one correspondence with this family of consistent finite-dimensional distributions, and hence 
there is a natural bijection between $\lim_{\leftarrow}\mathcal P(\bar{\mathbb R}^J)$ and $\mathcal P(\bar{\mathbb R}^I)$. 

We equip the space $\mathcal P(\bar{\mathbb R}^I)$ with the topology of \textit{weak convergence of the finite dimensional distributions},  
 which we define as follows:
 \begin{Definition}
 The topology of convergence of the finite dimensional distributions on $\mathcal P(\bar{\mathbb R}^I)$ is defined as the topology such that  $\mathcal P(\bar{\mathbb R}^I)\cong \lim_{\leftarrow}\mathcal P(\bar{\mathbb R}^J)$.
 \end{Definition}
 $\mathcal P(\bar{\mathbb R}^I)$ with this topology is by definition a topological inverse limit.
Define also $\lim_{\leftarrow}\mathcal C(\bar{\mathbb R}^J):=\lim_{\leftarrow}\mathcal P (\bar{\mathbb{R}}^J)\cap \prod_{J\in\mathcal{I}}\mathcal C (\bar{\mathbb{R}}^J)$. Certainly, we have \begin{equation}
    \mathcal C\left(\bar{\mathbb{R}}^I\right)\cong\lim_{\leftarrow}\mathcal C\left(\bar{\mathbb R}^J\right)
\end{equation}
with the corresponding topologies. 

The following result contains among other things 
the topological proof of Sklar's theorem \ref{T: Sklar in infinite dimensions}.
\begin{Theorem}\label{L: Compactness of Consistent Copula families}The following statements hold.
\begin{enumerate}
\item\label{Hausdorffness} $\mathcal{P}(\bar{\mathbb R}^I)$ with the topology of weak convergence of the finite dimensional distributions is a Hausdorff space
    \item\label{Compactness of the copulas} The space of consistent copulas $\mathcal C(\bar{\mathbb{R}}^I)$ is compact with respect to the topology of convergence of finite dimensional distributions. 
    \item\label{Second part of Sklar}
For a copula measure  $C$  on $\bar{\mathbb{R}}^I$ and (one-dimensional) Borel probability measures $(\mu_i)_{i\in I}$ over $\bar{\mathbb R}$ the push-forward measure
\begin{equation}\label{Explicit form of seond part Sklar}
\mu:=((F_{\mu_i}^{[-1]})_{i\in I})_*C
\end{equation}
satisfies 
\eqref{Sklar Property}. 
    \item\label{Continuity of Sklar} 
    If we equip $\mathcal C(\bar{\mathbb{R}}^I)\times \prod_{i\in I}\mathcal P (\bar{\mathbb R})$ with the product topology of weak convergence on each $\mathcal P (\bar{\mathbb R})$ and the topology of convergence of the finite dimensional distributions on $\mathcal C(\bar{\mathbb{R}}^I)$ and $\mathcal P(\bar{\mathbb R}^I)$, then the
   mapping $\Phi:\mathcal{C} (\bar{\mathbb{R}}^{\mathcal{I}})\times \prod_{i\in I} \mathcal{P}(\bar{\mathbb R})\to \mathcal{P}(\bar{\mathbb{R}}^I)$ given by  $$\Phi(C,(\mu_i)_{i\in I}):= ((F_{\mu_i}^{[-1]})_{i\in I})_*C$$
 is continuous and surjective. In particular, Sklar's theorem holds.
\end{enumerate}
\end{Theorem}
\begin{proof}

\eqref{Hausdorffness} Since products of Hausdorff spaces are Hausdorff and $\mathcal{P}(\bar{\mathbb R}^I)$ is homeomorphic to a subset of a product of Hausdorff spaces, it is Hausdorff.

\eqref{Compactness of the copulas} We know by \cite[Thm. 3.3]{MR2847456} that every $\mathcal C(\bar{\mathbb R}^J)$ is compact with respect to the topology of weak convergence on $\mathcal P(\bar{\mathbb{R}}^J)$.
 Tychonoff's Theorem guarantees also that $\prod_{J\in\mathcal{I}}\mathcal C(\bar{\mathbb R}^J)$ is compact with respect to the product topology on $\prod_{J\in\mathcal{I}}\mathcal P(\bar{\mathbb{R}}^J)$. Therefore, as $\lim_{\leftarrow}\mathcal P (\bar{\mathbb{R}}^J)$ is closed by Lemma \ref{L: Closedness of the Inverse limit}, we obtain that $\mathcal C (\bar{\mathbb{R}}^I)$ is compact, since it is homeomorphic to an intersection of a closed  and a compact set in the product topology. 

\eqref{Second part of Sklar}
This corresponds to the second part of Sklar's theorem and the proof can be conducted analogously to the one in  \cite{Benth2020}. Therefore, it is enough to see that $$\left(\left[0,F_{\mu_{j}}(x_j)\right]\right)_{j\in J}\setminus \left(\left(F_{\mu_j}^{[-1]}\right)^{-1}(-\infty,x_1]\right)_{j\in J}$$
is a $C_J$-nullset for all $(x_j)_{j\in J}\in\bar{\mathbb R}^J$, $J\in \mathcal I$, since then we immediately obtain
\begin{align*}
  C_J\left(\left(\left(F_{\mu_j}^{[-1]}\right)^{-1}(-\infty,x_1]\right)_{j\in J}\right)
  =  C_J\left(\left([0,F_{\mu_j}(x_j)]\right)_{j\in J}\right) =F_{C_J}\left(F_{\mu_j}\left(\left(x_j\right)\right)_{j\in J}\right).
\end{align*}

\eqref{Continuity of Sklar}
Define $\phi_J:\mathcal C(\bar{\mathbb R}^I)\times \prod_{i\in I}\mathcal P(\bar{\mathbb R})\to \mathcal P(\bar{\mathbb R}^J)$ by $$\phi_J(C,(\mu_i)_{i\in I}):=\Phi(C,(\mu_i)_{i\in I})_J,$$
which is well defined by \eqref{Second part of Sklar}. Since the finite-dimensional distributions of a law are consistent, $(\phi_J,J\in \mathcal I)$ forms a compatible family.
Define analogously for $J\in \mathcal I$ also $\tilde{\phi}_J:\mathcal C(\bar{\mathbb R}^J)\times \prod_{j\in J}\mathcal P(\bar{\mathbb R})\to \mathcal P(\bar{\mathbb R}^J)$  by $$\tilde{\phi}_J(C_J,(\mu_j)_{j\in J})= (F_{\mu_j}^{[-1]})_{j\in J})_*C_J.$$ 
This is by Sklar's theorem in finite-dimensions surjective and by \cite[Thm. 2]{MR2065562} also continuous.
Hence $\phi_J=\tilde{\phi}_J\pi_J$ is continuous and surjective, since both, $\tilde{\phi}_J$ and $\pi_J$ are. $\Phi$ must be the uniquely induced continuous mapping by the family $(\phi_J,J\in \mathcal I)$ by the universality property of the inverse limit.
Moreover, since by \cite[Corollary 4.2.6]{Schweitzer2006} $\mathcal P (\bar{\mathbb R})$ is compact and by \eqref{Compactness of the copulas} also $\mathcal C(\bar{\mathbb R}^I)$ is compact, we have that $\mathcal C(\bar{\mathbb R}^I)\times \prod_{i\in I}\mathcal P(\bar{\mathbb R})$ is compact by Tychonoff's theorem. The continuity of $\Phi$ implies therefore that $\Phi(\mathcal C(\bar{\mathbb R}^I)\times \prod_{i\in I}\mathcal P(\bar{\mathbb R}))$ is compact, hence closed.
Since moreover Lemma \ref{L: Surjectivity of the induced mapping} implies that $\Phi(\mathcal C(\bar{\mathbb R}^I)\times \prod_{i\in I}\mathcal P(\bar{\mathbb R}))$ is dense, we obtain that $\Phi$ is surjective and therefore also the first part of Sklar's theorem holds.
The uniqueness of the copulas in the case of continuous marginals follows immediately by Sklar's theorem in finite dimensions via the uniqueness of the finite dimensional distribution of the corresponding copula measure.
\end{proof}

Observe that since $\mathcal P (\bar{\mathbb R}^J)$ is a locally convex Hausdorff space with respect to the topology of weak convergence for each $J\in \mathcal I$, we obtain that also the inverse limit $\mathcal P(\bar{\mathbb R}^I)$ is locally convex, as it is isomorphic to a subset of the product $\prod_{J\in \mathcal I}\mathcal{P}(\mathbb R^J)$ of locally convex Hausdorff spaces.
Hence, as mentioned for instance in  \cite[p.30]{Sempi2015}, since $\mathcal C (\mathbb R^I)$ is convex, we have that it is the closure of its extremal points by the Krein-Milman theorem. As mentioned in \cite{Benes1991} this implies that 
$$\sup_{C\in \mathcal C(\mathbb R^I)}g(C)=\sup_{C\in  ext(\mathcal C(\mathbb R^I))}g(C)$$
where $ext(C(\mathbb R^I))$ denotes the set of extremal points of $C(\mathbb R^I)$ and $g:\mathcal C (\mathbb R^I)\mapsto \mathbb R$ is a convex function.

\subsection*{Acknowledgements}
This research was funded within the project STORM: Stochastics for Time-Space Risk Models, from the Research Council of Norway (RCN). Project number: 274410.

\bibliographystyle{amsplain}
\addcontentsline{toc}{section}{\refname}\bibliography{Bibliography}

\end{document}